\documentclass[12pt,leqno,a4paper]{amsart}
\usepackage{amssymb,enumerate}
\usepackage{enumitem}

\newcommand{\FF}{{\mathbb{F}}}

\newcommand{\bC}{{\mathbf{C}}}
\newcommand{\bH}{{\mathbf{H}}}
\newcommand{\bN}{{\mathbf{N}}}
\newcommand{\bT}{{\mathbf{T}}}

\newcommand{\fA}{{\mathfrak{A}}}
\newcommand{\fS}{{\mathfrak{S}}}

\newcommand{\Irr}{{\operatorname{Irr}}}
\newcommand{\Out}{{\operatorname{Out}}}
\newcommand{\OO}{{\operatorname{O}}}
\newcommand{\PGL}{{\operatorname{PGL}}}
\newcommand{\PSL}{{\operatorname{L}}}
\newcommand{\PSU}{{\operatorname{U}}}
\newcommand{\PGU}{{\operatorname{PGU}}}
\newcommand{\Syl}{{\operatorname{Syl}}}

\newcommand{\tw}[1]{{}^#1\!}

\newtheorem{thm}{Theorem}[section]
\newtheorem{lem}[thm]{Lemma}
\newtheorem{prop}[thm]{Proposition}

\newtheorem*{thmA}{Theorem A}

\theoremstyle{remark}

\newtheorem{exmp}[thm]{Example}

\begin{document}

\title[Defect zero characters predicted by local structure]{Defect zero characters\\ predicted by local structure}

\date{\today}

\author{Gunter Malle}
\address{FB Mathematik, TU Kaiserslautern, Postfach 3049,
        67653 Kaisers\-lautern, Germany.}
\email{malle@mathematik.uni-kl.de}
\author{Gabriel Navarro}
\address{Departament d'\`Algebra, Universitat de Val\`encia, 46100 Burjassot,
        Val\`encia, Spain.}
\email{gabriel@uv.es}
\author{Geoffrey R. Robinson}
\address{Institute of Mathematics, University of Aberdeen, Aberdeen AB24 3UE, Scotland, United Kingdom}
\email{g.r.robinson@abdn.ac.uk}

\thanks{The first author gratefully acknowledges financial support by ERC
Advanced Grant 291512. The second author is partially supported by the
Spanish Ministerio de Educaci\'on y Ciencia Proyectos MTM2013-40464-P and
Prometeo II/Generalitat Valenciana.}



\begin{abstract}
Let $G$ be a finite group and let $p$ be a prime. Assume that there exists
a prime $q$ dividing $|G|$ which does not divide the order of any $p$-local
subgroup of $G$. If $G$ is $p$-solvable or $q$ divides $p-1$, then $G$ has
a $p$-block of defect zero. The case $q=2$ is a well-known result by Brauer
and Fowler.
\end{abstract}

\maketitle

%%\pagestyle{myheadings}
%%\markboth{}{}
%%\markboth{for personal use only}{preliminary}

%%%%%%%%%%%%%%%%%%%%%%%%%%%%%%%%%%%%%%%%%%%%%%%%%%%%%%%%%%%%%%%%%%%%%%%%%
\section{Introduction}   \label{sec:intro}

Let $G$ be a finite group and let $p$ be a prime.
A $p$-defect zero character of $G$ is an irreducible complex character
$\chi\in\Irr(G)$ whose degree has $p$-part $\chi(1)_p=|G|_p$. For several
well-known reasons, $p$-defect zero characters play an important role in
Representation Theory, and they are the subject of key questions in this
field by Richard Brauer (as Problem 19 of \cite{Br}, solved by the third author
in \cite{R}) or Walter Feit (Problem VI, Chapter IV of \cite{F}).

\medskip

This note is a small contribution to Feit's Problem VI, in which he asks about
necessary and sufficient conditions for the existence of characters of defect
zero. There are too many results in this area, by R.~Brauer and K.~A.~Fowler,
N.~Ito, G.~R.~Robinson, Y.~Tsushima, T.~Wada and others, to be listed here.
It is elementary to show, however, that a necessary condition for the existence
of a $p$-defect zero character in a finite group $G$ is that the group has no
non-trivial normal $p$-subgroups.
In this note, we consider local subgroups (normalisers of non-trivial
$p$-subgroups), and some special primes to give a sufficient condition.

\begin{thmA}
 Let $G$ be a finite group and let $p$ and $q$ be different primes dividing
 $|G|$. If $G$ is not $p$-solvable, assume that $q|(p-1)$.
 If $q$ does not divide $|\bN_G(R)|$ for every $p$-subgroup $R>1$ of $G$, then
 $G$ has a $p$-defect zero character.
\end{thmA}

We point out that the proof of Theorem A, if $G$ is not $p$-solvable, uses
the Classification of Finite Simple Groups. In the case where $q=2$, however,
Theorem~A follows from a classical result by Brauer and Fowler \cite[(5F)]{BF}.
In fact, as in the Brauer--Fowler theorem, we  only need to consider subgroups
$R$ of $G$ of order $p$, and therefore our main result can be regarded as a
version of Brauer-Fowler for odd primes.

As a by-product of our proof in Theorem~\ref{thm:no defect zero} we also
classify for $p\ge5$ the almost simple groups $G$ with $G/G'$ cyclic of prime
power order $p^a$ having no $p$-defect zero characters. The analogous
classification in the case when $p\le3$ seems somewhat more involved.

%%%%%%%%%%%%%%%%%%%%%%%%%%%%%%%%%%%%%%%%%%%%%%%%%%%%%%%%%%%%%%%%%%%%%%%%%
\section{Main Results}

As we have pointed out in the introduction, for primes $q$ dividing $p-1$,
the hypothesis that $q$ does not divide $|\bN_G(R)|$ for every $p$-subgroup
$R>1$ of $G$ is equivalent to the condition that $q$ does not divide
$|\bN_G(R)|$ for every subgroup $R$ of $G$ of order $p$.
This is a consequence of the following.

\begin{lem}   \label{same}
 Suppose that a non-trivial $q$-group $Q$ acts by automorphisms on a
 non-trivial $p$-group $P$, where $p$ and $q$ are primes such that $q$
 divides $p-1$. Then there exists $x \in P$ of order $p$ such that
  $|\bN_Q(\langle x\rangle)|>1$.
\end{lem}

\begin{proof}
We argue by induction on $|P|$. We may assume that $|Q|=q$. If $1<N <P$ is
$Q$-invariant, then we are done by induction. So we may assume that $P$ is an
irreducible $Q$-module. Also, we may assume that $P$ is faithful.
But in this case, it is well-known that $P$ can be identified with
$V=\FF_{p^m}$ and that $Q \subseteq \FF_{p^m}^\times$ acts by multiplications.
(See for instance  of \cite[Thm.~2.1]{MW}.) Let $Q=\langle y\rangle$.
Hence, if $0\ne v \in V$ then $v^y=\lambda v$ for some
$\lambda \in \FF_{p^m}^\times$. Now $\FF_{p^m}^\times$ is cyclic, and
it has a unique subgroup of order $q$ that lies in $\FF_p^\times$, using that
$q$ divides $p-1$. We deduce that $\lambda \in \FF_p^\times$, and
$v^y \in \langle v\rangle$. Thus $|\bN_Q(\langle v\rangle)|>1$.
\end{proof}

Of course, Lemma~\ref{same}, is no longer true without the hypothesis
of $q$ dividing $p-1$, as the alternating group $\fA_4$ shows.

\medskip

Now, assuming the result of Theorem~\ref{thm:almost} on almost simple groups
whose proof we defer to the next section, we proceed to prove our Theorem A
which we restate for the reader's convenience:

\begin{thm}
 Let $G$ be a finite group and let $p$ and $q$ be different primes dividing
 $|G|$. If $G$ is not $p$-solvable, assume that $q|(p-1)$.
 If $q$ does not divide $|\bN_G(R)|$ for every $p$-subgroup $R>1$ of $G$, then
 $G$ has a $p$-defect zero character.
\end{thm}

\begin{proof}
We argue by induction on $|G|$. Clearly, $O_p(G)=1$. Notice that our hypotheses
are inherited by the
subgroups of $G$ of order divisible by $q$. In particular, if $K$ is a proper
normal subgroup of $G$ of order divisible by $q$ with $p'$-index, then we
conclude that $K$ has an irreducible $p$-defect zero character
$\theta\in\Irr(K)$. Now, if $\chi \in \Irr(G)$ lies over $\theta$, then
$\chi$ has $p$-defect zero and $G$ is not a counterexample.

Let $M$ be a minimal normal subgroup of $G$.

Suppose first that $M$ is a $p'$-group. If $R > 1$ is a $p$-subgroup of $G$,
then $\bN_{G/M}(RM/M)=\bN_G(R) M/M$
is isomorphic to $\bN_G(R)/\bN_M(R)$, and therefore has order not divisible
by $q$. Thus if $q$ divides $|G/M|$, then by induction, $G/M$ has a
$p$-defect zero character, which is a $p$-defect zero character of $G$.
Hence, we may assume in this case that $G/M$ is a $q'$-group.
Now, let $x \in M$ be of order $q$. Then by hypothesis, $|\bC_G(x)|$ is $p'$.
Hence $G$ has a $p$-block of defect zero by \cite[Thm.~1]{T}.

Hence, we may assume that $M$ has order divisible by $p$.
By the Frattini argument and our hypothesis, we have that $G/M$ has order not
divisible by $q$. So $q$ divides $|M|$. Thus, $M$ is isomorphic
to a direct product of copies of a non-abelian simple group $S$ of order
divisible by $pq$. Since $G$ has no elements of order $pq$, $M$ is simple,
and is the unique minimal normal subgroup of $G$. Thus $G$ is almost simple.
Also, notice now that $O^{p'}(G)=G$, by the first paragraph of this proof.
So we conclude that $G/G'$ is a $p$-group. Now, let $Q \in\Syl_q(M)$ and
$P\in\Syl_p(\bN_G(Q))$. If $P$ is not cyclic, then it has an
elementary abelian $p$-subgroup $A$ of order $p^2$, for instance, using
\cite[Thm.~5.4.10]{Go}. Then $Q=\langle \bC_Q(x) \mid  1\ne x \in A \rangle$,
by \cite[Thm.~5.3.16]{Go}. By hypothesis $\bC_Q(x)=1$ for all $x\neq1\in A$,
and this is a contradiction. Hence $G/M$ has cyclic Sylow $p$-subgroups.
According to Theorem~\ref{thm:almost}, $G$ is then not a counterexample.
\end{proof}

\begin{exmp}
 The condition that $q$ divides $p-1$ is necessary, as shown by the Mathieu
 group $M_{22}$: None of its proper $2$-local subgroups has order divisible
 by $q = 11$, but still it does not possess characters of 2-defect zero.
\end{exmp}

Theorem~\ref{thm:no defect zero} gives further examples.

%%%%%%%%%%%%%%%%%%%%%%%%%%%%%%%%%%%%%%%%%%%%%%%%%%%%%%%%%%%%%%%%%%%%%%%%%
\section{Almost Simple Groups}

\begin{thm}   \label{thm:almost}
 Let $p$ be a prime and $G$ be a finite almost simple group such that $G/G'$
 is cyclic of prime power order $p^a$. Then one of:
 \begin{enumerate}
  \item[\rm(1)] $G$ has an irreducible character of $p$-defect zero, or
  \item[\rm(2)] for every prime $r$ dividing $p-1$ there is a cyclic subgroup
   $U\cong C_p$ of $G$ with $r$ dividing $|\bN_G(U)|$.
 \end{enumerate}
\end{thm}

Note that (1) can fail to hold. See Theorem~\ref{thm:no defect zero} below
for a precise statement.  \par
We will subdivide the proof of Theorem~\ref{thm:almost} into several steps.
First note that the assertion~(2) is vacuously satisfied when $p=2$. So we
may assume that $p\ge3$. Now according to the classification of finite simple
groups a finite non-abelian simple group $S$ with an outer automorphism
of odd prime order must be of Lie type. \par
Also note that for $p=3$ we only need to consider $r=2$, and in
particular the result holds if $G$ has an element of order~6.
Finally note that Theorem~\ref{thm:almost} holds when $p^a=1$, that is,
when $G$ is simple, by Michler's theorem \cite{Mi87}.

There are two quite different cases which arise: if $p$ is the defining prime
of a group of Lie type, then we show that (2) holds (while (1) will fail in
general, see Theorem~\ref{thm:no defect zero}).
The same happens when $G/G'$ is generated by diagonal automorphisms.
In all the remaining cases when $p$ is not the defining prime we argue that
there always exists a $p$-defect zero character. This result may be of
independent interest.

We start off with a preliminary reduction that restricts the type of almost
simple groups we need to look at.

\begin{lem}   \label{lem:out}
 Let $G$ be a finite almost simple group with simple socle $S$ such that
 $G/G'$ is cyclic of prime power order $p^a>1$ for some odd prime $p$. Then
 $S$ if of Lie type and one of the following holds:
 \begin{enumerate}
  \item[\rm(1)] $G/G'$ is generated by a field automorphism and
   \begin{enumerate}
    \item[\rm(a)] either $G'=S$, or
    \item[\rm(b)] $S=\PSL_n(q)$ or $\PSU_n(q)$, and $G'/S$ is generated by a
     diagonal automorphism of order prime to $p$;
   \end{enumerate}
  \item[\rm(2)] $S=\PSL_n(q)$ or $\PSU_n(q)$, $p|(n,q\pm1)$, and $G/S$ is
   generated by a diagonal automorphism; or
  \item[\rm(3)] $p=3$ and $S$ is of type $D_4$, $E_6$ or $\tw2E_6$.
 \end{enumerate}
\end{lem}

\begin{proof}
As the outer automorphism groups of alternating and sporadic groups are
2-groups, $S$ must be of Lie type. Let $H:=\Out(S)$. Then $H$ has a quotient
isomorphic to the direct product of the group of field automorphisms times
the group of graph automorphisms of $S$. Moreover, graph automorphisms have
order at most~3, and 3 only occurs for type $D_4$. Hence, if $G$ involves
graph automorphisms, then $p=3$ and we are in case~(3). \par
Now assume that no graph automorphisms are present. The group of diagonal
automorphisms is a 2-group unless we are in types $A_n$, $\tw2A_n$ or $E_6$,
$\tw2E_6$. In the latter case, diagonal automorphisms have order dividing~3,
hence again we are in case~(3). In the first case, as we have no graph
automorphisms in $G$, $G/S$ is a semidirect product of the cyclic group of
diagonal automorphisms with the cyclic group of field automorphisms. If no
field automorphisms are present we arrive at case~(2),
if no diagonal automorphisms are present we get~(1)(a), and if both occur,
then since $G/G'$ is cyclic of $p$-power order, the group of diagonal
automorphisms must be prime to $p$, so we arrive at~(1)(b).
\end{proof}

\begin{prop}   \label{prop:defchar}
 Let $G$ be as in Theorem~\ref{thm:almost} with simple socle $S$ of Lie type
 such that $p$ is the defining prime for $S$. Then for every prime divisor $r$
 of $p-1$ there is a subgroup $U\cong C_p$ of $G$ such that $r$ divides
 $|\bN_G(U)|$. \par
 In particular Theorem~\ref{thm:almost} holds in this case.
\end{prop}

\begin{proof}
Let $S$ be simple of Lie type in characteristic~$p$. Let $B$ be a Borel
subgroup of $S$. Then the order of $B$ is divisible by $p-1$, unless
either $S\cong\PSL_2(q)$ with $q=p^f$ odd, in which case $|B|$ is still
divisible by $(p-1)/2$, or $S\cong\PSU_3(q)$ with $q=p^f\equiv2\pmod3$, in
which case $|B|$ is divisible by $(p^2-1)/3$.
\par
In the first case, only $r=2$ may cause problems (if $q\equiv3\pmod4$). But
$S$ has a unique class of involutions, so any extension of degree $p$ contains
elements of order $rp=2p$. In the second case, only $r=3$ may cause problems.
But again $S$ has a unique class of elements of order~3, whence any extension
of degree~$p$ (which is at least~$5$ in this case) contains elements of
order~$pr$.
\end{proof}

We now consider the remaining possibilities according to Lemma~\ref{lem:out} in
the case that $p$ is not the defining prime.

\begin{prop}   \label{prop:diag}
 Let $G$ be as in Theorem~\ref{thm:almost} with simple socle $S=\PSL_n(q)$
 or $S=\PSU_n(q)$, and assume that $G/S$ is generated by a diagonal
 automorphism of $S$ of order $p^a$. Then there is a subgroup $U\cong C_p$
 of $S$ with $|\bN_G(U)|$ divisible by~$p-1$.  \par
 In particular Theorem~\ref{thm:almost} holds in this case.
\end{prop}

\begin{proof}
By assumption $S$ has a diagonal outer automorphism of prime order $p$, so
$p|(n,q-1)$ if $S$ is a linear group $\PSL_n(q)$, and $p|(n,q+1)$ if $S$ is
a unitary group $\PSU_n(q)$. In either case, $p$ divides the order of the Weyl
group $\fS_n$ of $G$. But elements of $\fS_n$ are rational, so there is a
subgroup $U\cong C_p$ of $G$ with normaliser order divisible by $p-1$.
\end{proof}

\begin{lem}
 Theorem~\ref{thm:almost} holds when $p=3$ and $G$ is of type $D_4$, $E_6$ or
 $\tw2E_6$.
\end{lem}

\begin{proof}
This is immediate since here necessarily $r=2$ and any simple group of the
listed types contains elements of order~$6=pr$.
\end{proof}

\begin{thm}   \label{thm:non-def}
 Let $G$ be as in Theorem~\ref{thm:almost} with simple socle $S$ of Lie type,
 $p>2$ not the defining prime for $S$, and $G/G'$ generated by field
 automorphisms. Then $G$ has an irreducible character of $p$-defect zero.
 In particular Theorem~\ref{thm:almost} holds in this case.
\end{thm}

\begin{proof}
Note $G$ has a character of $p$-defect zero if $S$ has a character $\chi$ of
$p$-defect zero with inertia group $I_G(\chi)$ such that $|I_G(\chi):S|$ is
prime to $p$. So, according to the description of the structure of
$G/S$ in Lemma~\ref{lem:out} our task is to find $p$-defect zero characters
of $S$ with trivial inertia group under field automorphisms. As already
for the case that $G=S$ is simple in \cite{Mi87}, the idea how to produce
such characters is to use irreducible Deligne--Lusztig characters with
respect to suitable maximal tori.

We work in the following setting. Let $\bH$ be a simple algebraic group in
characteristic not $p$ of simply connected type with a Steinberg endomorphism
$F:\bH\rightarrow\bH$ such that $H=\bH^F$ is quasi-simple with $S=H/Z(H)$.
Assume that $\gamma$ is a field automorphism of $S$ of order a positive power
of $p$. Thus in particular the underlying prime power $q$ of $H$ is a $p$th
power and our group $S$ is not a group over the prime field (and hence in
particular $S\ne\tw2F_4(2)'$).

Now let $\bT_1,\bT_2$ be $F$-stable maximal tori of the dual group $\bH^*$ such
that $T_i=\bT_i^F$ are as in
Tables~1 and~2 of \cite{Ma10}. It is argued in \cite[Prop.~2.4 and~2.5]{Ma10}
that both maximal tori $T_i$ do contain regular elements of $\bH^{*F}$
in most cases, in fact always if $H$ is not over the prime field and not of
types $A_1,A_2,\tw2A_2$ or $B_2$, and there even exist elements $s_i\in T_i$
whose $d$th power is regular, where $d=|Z(H)|=\gcd(|T_1|,|T_2|)$. In particular,
we can take for $s_i$ an element of maximal order in $T_i$, since if $s_i$ is
regular then so is any root of it. But then by order reasons $s_i^d$ cannot
lie in a subfield subgroup of $\bH^{*F}$ (which is defined over a field of
cardinality at most $\sqrt[3]{q}$). Hence $s_i^d$ is not fixed by any field
automorphism of $G$ of order $p$. But then the corresponding Deligne--Lusztig
character $R_{T_i}(s_i^d)$ is irreducible (up to sign), of $p$-defect zero, and
not invariant under field automorphisms of order~$p$ and we are done.
\par
For the four excluded series $A_1,A_2,\tw2A_2$ or $B_2$ the assertion is
easily checked directly, again along the lines of the argument given in
\cite[Prop.~2.5]{Ma10}.
\end{proof}

Observe that the proof of Theorem~\ref{thm:almost} is now complete. We close
by classifying those almost simple groups $G$ as in Theorem~\ref{thm:almost}
violating conclusion~(1), so having no $p$-defect zero character for $p\ge5$.

\begin{thm}   \label{thm:no defect zero}
 Let $G$ be as in Theorem~\ref{thm:almost} with simple socle $S$. Assume that
 $p\ge5$. Then $G$ has no $p$-block of defect zero if and only if $G\ne S$
 and one of the following holds:
 \begin{enumerate}
  \item[\rm(1)] $S$ is of Lie type in characteristic~$p$, or
  \item[\rm(2)] $S\le G\le\PGL_n(q)$ or $S\le G\le\PGU_n(q)$ with $p$ dividing
   $|G:S|$.
 \end{enumerate}
\end{thm}

\begin{proof}
The ``only if'' part is a consequence of Lemma~\ref{lem:out} in conjunction
with Theorem~\ref{thm:non-def} and the theorem of Michler (when $G=S$).
As for the ``if'' part, note that by \cite[Thm.~1.1]{MZ01}, $S$ in
characteristic~$p$ has a unique $p$-defect zero character, viz.~the Steinberg
character, which is hence invariant under all outer automorphisms, so $G$
cannot have defect zero characters when $G>S$.
\par
Now assume that $S\le G\le\PGL_n(q)$ with $p$ dividing $|G:S|$. There is a
simple algebraic group $\bH$ of type $A_{n-1}$ with a Frobenius map $F$ such
that $G\cong H/Z(H)$, where $H=\bH^F$. Now assume that $\chi\in\Irr(G)$ has
$p$-defect zero. Then we may consider $\chi$ as an irreducible character of
$H$ of central defect. In particular, the degree polynomial of $\chi$ is
divisible by $(q-1)^n$ (as $p$ divides $q-1$). By Lusztig's Jordan
decomposition this means that $\chi$ lies in a Lusztig series of a semisimple
element $s\in \bH^{*F}$ such that the $\Phi_1$-torus of $C_{\bH^*}(s)$
lies in $Z(\bH^*)$. But this implies that $C_{\bH^*}(s)$ is a Coxeter torus
$\bT$ of $\bH^*$, and thus $\chi$ is (up to sign) an irreducible
Deligne--Lusztig character of $H$, of degree $|\bH^{*F}:\bT^F|$, which cannot
be of central defect as the $p$-part $|\bT^F|_p=(q^n-1)_p$ is larger than
$|Z(H)|_p$ (because $p|(q^n-1)/(q-1)$ as $p$ divides $|G:S|$ which in turn
divides $n$). \par
The argument for the unitary case is completely analogous.
\end{proof}

For $p=3$ there are further examples without characters of $3$-defect zero
when $S=\OO_8^+(q)$, $E_6(q)$ or $\tw2E_6(q)$. Even more examples exist when
$p=2$.

\begin{exmp}
 It is well known that the symmetric group $\fS_n$ has a 2-block of defect
 zero if and only if $n$ is a triangular number. Now if $q \le n$ is an odd
 prime, the only time that all $2$-locals of $\fS_n$ are of $q'$-order is
 when $n = q$, or $n = q+1$ if $q$ is not a Mersenne prime.

 But the only odd prime which is a triangular number is $q = 3$ and the only
 time that $q + 1$ is a triangular number $m(m+1)/2$ is when $2q = (m+2)(m-1)$,
 which only happens for $m = 3$, so $q = 5$. Hence for any prime $q >5$,
 the symmetric group $\fS_q$ (and $\fS_{q+1}$ if $q$ is not a Mersenne prime)
 has no $2$-block of defect zero, and no $2$-local subgroup of order divisible
 by $q$. This gives infinitely many counterexamples to the conclusion of
 Theorem~A if we drop the condition that $q |(p -1)$.
\end{exmp}

%%%%%%%%%%%%%%%%%%%%%%%%%%%%%%%%%%%%%%%%%%%%%%%%%%%%%%%%%%%%%%%%%%%%%%%%%

\end{document}